\documentclass[12pt,a4paper]{tac}
\usepackage{amssymb, amsmath, stmaryrd, mathrsfs}
\usepackage[latin1]{inputenc}
\usepackage[arrow, matrix, tips, curve]{xy}
\SelectTips{cm}{10}
\usepackage{enumitem}
\setlist[enumerate]{noitemsep}
\setlist[enumerate,1]{label=(\alph*), ref=(\alph*)}
\setlist[enumerate,2]{label=(\roman*),
ref=\theenumi(\roman*)}

\DeclareMathOperator{\ob}{ob}

\newcommand{\cat}[1]{\mathbf{#1}} 

\newcommand{\op}{\mathrm{op}} 

\newcommand{\thg}{{\mathord{\text{--}}}}

\newcommand{\abs}[1]{{\left|{#1}\right|}}

\newcommand{\spn}[1]{{\left<{#1}\right>}}

\newcommand{\cd}[2][]{\vcenter{\hbox{\xymatrix#1{#2}}}}
\newcommand{\cdl}[2][]{\xymatrix@1#1{#2}}

\newcommand{\A}{{\mathcal A}}
\newcommand{\B}{{\mathcal B}}
\newcommand{\C}{{\mathcal C}}
\newcommand{\D}{{\mathcal D}}
\newcommand{\E}{{\mathcal E}}

\newcommand{\M}{{\mathcal M}}

\newcommand{\V}{{\mathcal V}}
\newcommand{\W}{{\mathcal W}}

\newcommand{\xtor}[1]{\cdl[@1]{{} \ar[r]|-{\object@{|}}^{#1} & {}}}
\newcommand{\tor}{\ensuremath{\relbar\joinrel\mapstochar\joinrel\rightarrow}}

\newcommand{\twocong}[2][0.5]{\ar@{}[#2] \save ?(#1)*{\cong}\restore}
\newcommand{\rtwocell}[3][0.5]{\ar@{}[#2] \ar@{=>}?(#1)+/l 0.2cm/;?(#1)+/r 0.2cm/^{#3}}
\newcommand{\rtwocello}[3][0.5]{\ar@{}[#2] \ar@{=>}?(#1)+/l 0.2cm/;?(#1)+/r 0.2cm/_{#3}}
\newcommand{\ltwocell}[3][0.5]{\ar@{}[#2] \ar@{=>}?(#1)+/r 0.2cm/;?(#1)+/l 0.2cm/^{#3}}
\newcommand{\ltwocello}[3][0.5]{\ar@{}[#2] \ar@{=>}?(#1)+/r 0.2cm/;?(#1)+/l 0.2cm/_{#3}}
\newcommand{\dtwocell}[3][0.5]{\ar@{}[#2] \ar@{=>}?(#1)+/u  0.2cm/;?(#1)+/d 0.2cm/^{#3}}
\newcommand{\dthreecell}[3][0.5]{\ar@{}[#2] \ar@3{->}?(#1)+/u  0.2cm/;?(#1)+/d 0.2cm/^{#3}}
\newcommand{\utwocell}[3][0.5]{\ar@{}[#2] \ar@{=>}?(#1)+/d 0.2cm/;?(#1)+/u 0.2cm/_{#3}}
\newcommand{\dtwocelltarg}[3][0.5]{\ar@{}#2 \ar@{=>}?(#1)+/u  0.2cm/;?(#1)+/d 0.2cm/^{#3}}
\newcommand{\utwocelltarg}[3][0.5]{\ar@{}#2 \ar@{=>}?(#1)+/d  0.2cm/;?(#1)+/u 0.2cm/_{#3}}

\newdir{ >}{{}*!/-5pt/\dir{>}}

\newtheorem{Prop}{Proposition}

\newtheorem{Thm}[Prop]{Theorem}

\makeatletter

\renewcommand{\paragraph}{\@startsection
{paragraph}%
{3}%
{0mm}%
{-\baselineskip}%
{-0.4em plus 0.2em minus 0.2em}%
{\normalfont\normalsize\bfseries}}%
\makeatother

\makeatletter \@namedef{itemize*}{\itemize\parsep\z@ \parskip\z@}
\@namedef{enditemize*}{\enditemize} \@namedef{enumerate*}{\enumerate\parsep\z@
\parskip\z@} \@namedef{endenumerate*}{\endenumerate}
\def\Pr@@f{\subsubsection*{\textbf{Proof}}}
\def\pr@@f[#1]{\subsubsection*{{\textbf{Proof}} #1}}
\makeatother
\DeclareMathAlphabet      {\mathbf}{OT1}{cmr}{b}{n}

\begin{document}

\title{Diagrammatic characterisation of\\enriched absolute colimits}
\author{Richard Garner}
\date{\today}

\copyrightyear{2014}
\maketitle
\begin{abstract}
We provide a diagrammatic criterion for the existence of an absolute
colimit in the context of enriched category theory.
\end{abstract}
\vskip\baselineskip An \emph{absolute colimit} is one preserved by any
functor; the class of absolute colimits was characterised for ordinary
categories by Par\'e~\cite{Pare1971On-absolute} and for enriched ones
by Street~\cite{Street1983Absolute}. For categories enriched over a
monoidal category $\V$ or bicategory $\W$, the appropriate colimits
are the weighted colimits of~\cite{Street1983Enriched}, and Street's
characterisation is in fact one of the class of \emph{absolute
  weights}: those weights $\varphi$ such that $\varphi$-weighted
colimits are preserved by any functor. This is different to Par\'e's
result, which gives a diagrammatic characterisation of when a
particular cocone is absolutely colimiting. In this note, we give
a result in the enriched context which is closer in spirit to Par\'e's
than to Street's. This result is very useful in practice, but seems not
to be in the literature; we set it down for future use.

\section{The result}
\subsection{Background} We work in the context of bicategory-enriched
category theory; see~\cite{Street1983Enriched}, for example. $\W$ will
denote a bicategory whose homs are locally small, complete and
cocomplete categories, and which is \emph{biclosed}, meaning that for
each $1$-cell
$A \colon x \to y$ in $\W$, the composition functors $A \otimes (\thg)
\colon \W(z,x) \to \W(z,y)$ and $(\thg) \otimes A \colon \W(y,z) \to
\W(x,z)$ have right adjoints $[A, \thg]$ and $\spn{A, \thg}$
respectively.

A \emph{$\W$-category}
$\A$ comprises a set $\ob \A$ of objects; for each $a \in \ob \A$ an
object $\epsilon a \in \ob \W$, the \emph{extent} of $a$; for each
pair of objects $a,b$, a hom-object $\C(b,a) \in \W(\epsilon a,
\epsilon b)$; and identity and composition $2$-cells $\iota \colon
I_{\epsilon a} \to \C(a,a)$ and $\mu \colon \C(c,b) \otimes \C(b,a)
\to \C(c,a)$ satisfying the expected axioms.
A $\W$-\emph{profunctor} $M \colon \A \tor \B$ is given by objects
$M(b,a) \in \W(\epsilon a, \epsilon b)$ and action maps $\mu \colon
\B(b',b) \otimes M(b,a) \otimes \A(a,a') \to M(b',a')$ satisfying
unitality and associativity axioms. A \emph{profunctor map} $M \to M'
\colon \A \tor \B$ comprises maps $M(b,a) \to M'(b,a)$ compatible with
the actions by $\A$ and $\B$. The identity profunctor $\A \colon \A
\tor A$ has components $\A(b,a)$ with action given by composition in
$\A$. For profunctors $M \colon \A \tor \B$ and $N \colon \B \tor \C$
with $\B$ small, the tensor product $N \otimes_\B M \colon \A
\tor \C$ has components given by coequalisers
\[
\textstyle \sum_{b,b'} N(c,b) \otimes \B(b,b') \otimes M(b', a) \rightrightarrows
\sum_b N(c,b) \otimes M(b,a) \twoheadrightarrow (N \otimes_\B M)(c,a)
\]
and actions by $\C$
and $\A$ inherited from $N$ and $M$. Small
$\W$-categories, profunctors and profunctor maps comprise a bicategory
$\W\text-\cat{Mod}$. There is a full embedding $\W \to
\W\text-\cat{Mod}$ sending $X$ to the $\W$-category $X$ with one
object $\star$ with $\epsilon(\star) = X$ and $X(\star, \star) = I_X$.

If $\A$ and $\B$ are $\W$-categories, then a \emph{$\W$-functor} $F
\colon \A \to \B$ comprises an extent-preserving assignation on
objects, together with $2$-cells $\C(b,a) \to \D(Fb,Fa)$
subject to two functoriality
axioms. 
If $F \colon \A \to \C$ and $G \colon \B \to \C$ are $\W$-functors
then there is an induced profunctor $\C(G, F) \colon \A \tor \B$ with
components $\C(G,F)(b,a) = \C(Gb,Fa)$ and action derived from the
action of $F$ and $G$ on homs and composition in $\C$.

Given profunctors $M \colon \A \tor \B$, $N \colon \B \tor \C$ and $L
\colon \A \tor \C$ with $\B$ small, a profunctor map $u \colon N
\otimes_\B M \to L$ is said to \emph{exhibit $M$ as $[N,L]$} if every
map $f \colon N \otimes_\B K \to L$ is of the form $u \circ (N
\otimes_\B \bar f)$ for a unique $\bar f \colon K \to M$; while it is
said to \emph{exhibit $N$ as $\spn{M,L}$} if every $f \colon K
\otimes_\B M \to L$ is of the form $u \circ (\bar f \otimes_\B M)$ for
a unique $\bar f \colon K \to N$.

Given $\varphi \colon \A \tor \B$ in $\W\text-\cat{Mod}$ and a functor
$F \colon \B \to \C$, a \emph{$\varphi$-weighted colimit} of $F$ is a
functor $Z \colon \A \to \C$ and profunctor map $a \colon \varphi
\to \C(F, Z)$ such that for each $C \in \C$, the map
\begin{equation}
\varphi \otimes_\A \C(Z,C) \xrightarrow{a \otimes_\A 1} \C(F,Z)
\otimes_\A \C(Z,C) \xrightarrow{\ \ \mu \ \ }
\C(F,C)\label{eq:1}
\end{equation}
exhibits $\C(Z,C)$ as $[\varphi, \C(F,C)]$. A functor $G \colon \C \to
\D$ \emph{preserves} this colimit just when the composite $\varphi \to
\C(F, Z) \to \D(GF, GZ)$ exhibits $GZ$ as a $\varphi$-weighted colimit
of $GF$; the colimit is \emph{absolute} when it is preserved by all
functors out of $\C$. \cite{Street1983Absolute} proves that
$\varphi$-weighted colimits are absolute if and only if $\varphi$
admits a right adjoint in $\W\text-\cat{Mod}$.

Dually, given $\psi \colon \B \tor \A$ in $\W\text-\cat{Mod}$ and a
functor $F \colon \B \to \C$, a \emph{$\psi$-weighted limit} of $F$
is a functor $Z \colon \A \to \C$ and map $b \colon \psi
\to \C(Z, F)$ such that for each $C \in \C$, the map
\begin{equation*}
\C(C, Z) \otimes_\A \psi \xrightarrow{1 \otimes_\A b} \C(C,Z) \otimes_\A
\C(Z,F) \xrightarrow{ \ \ \mu \ \ } \C(C,F)
\end{equation*}
exhibits $\C(C,Z)$ as $\spn{\psi, \C(C,Z)}$. Absoluteness of limits is
defined as before; now every limit weighted by $\psi \colon \B
\tor \A$ is absolute if and only if $\psi$ has a \emph{left} adjoint in
$\W\text-\cat{Mod}$.



\begin{Thm}
  \label{prop:1}
  Let $\varphi \colon \A \tor \B$ admit the right adjoint $\psi \colon
  \B \tor \A$ in $\W\text-\cat{Mod}$, and let $F \colon \B \to \C$ and
   $Z \colon \A \to \C$ be $\W$-functors. There is a bijective
  correspondence between data of the following forms:
  \begin{enumerate}
  \item A map $a \colon \varphi \to \C(F, Z)$ exhibiting $Z$ as a
    $\varphi$-weighted colimit of $F$;
  \item A map $b \colon \psi \to \C(Z, F)$ exhibiting $Z$ as a
    $\psi$-weighted limit of $F$;
  \item Maps $a \colon  \varphi \to \C(F, Z)$ and $b \colon \psi \to \C(Z,
    F)$ such that the following two squares commute in
    $\W\text-\cat{Mod}(\A, \A)$ and $\W\text-\cat{Mod}(\B, \B)$:
\begin{equation}
\cd{
\A \ar[r]^\eta \ar[d]_{Z} & \psi \otimes_\B \varphi \ar[d]^{b
  \otimes_\B a} \\
\C(Z,Z) \ar@{<-}[r]_-{\ \mu} & \C(Z,F) \otimes_\B \C(F,Z)
} \qquad \qquad
\cd{
\varphi \otimes_\A \psi \ar[r]^{\varepsilon} \ar[d]_{a \otimes_\A b} & \B \ar[d]^{F} \\
\C(F,Z) \otimes_\A \C(Z,F) \ar[r]_-{\mu} & \C(F,F)\rlap{ .}
}\label{eq:4}
\end{equation}
  \end{enumerate}
\end{Thm}
\begin{proof}
Suppose first given (a); consider the composite
  profunctor map
\begin{equation}
\varphi \otimes_\A \C(Z,F) \xrightarrow{a \otimes_\A 1} \C(F,Z)
\otimes_\A \C(Z,F) \xrightarrow{\ \ \mu \ \ }
\C(F,F)\rlap{ .}\label{eq:3}
\end{equation}
Evaluating in the second variable at any $a \in \A$ yields the
map~\eqref{eq:1} exhibiting $\C(Z,Fa)$ as $[\varphi, \C(F,Fa)]$; it
follows easily that~\eqref{eq:3} exhibits $\C(Z,F)$ as $[\varphi,
\C(F,F)]$. Applying this universality to the composite
$\varepsilon \circ F \colon \varphi \otimes_\A \psi \to \B \to
 \C(F,F)$
yields a unique map $b \colon \psi \to \C(Z,F)$ making the
right square of~\eqref{eq:4} commute; we must show that the left one
does too. Arguing as before shows that
\begin{equation}
  \varphi \otimes_\A \C(Z,Z) \xrightarrow{a \otimes_\A 1} \C(F,Z)
\otimes_\A \C(Z,Z) \xrightarrow{\ \ \mu \ \ }
\C(F,Z)\label{eq:5}
\end{equation}
exhibits $\C(Z,Z)$ as $[\varphi, \C(F,Z)]$. It thus suffices to show
that the left square of~\eqref{eq:4} commutes after applying the
functor $\varphi \otimes_\A (\thg)$ and postcomposing
with~\eqref{eq:5}; which follows by a short calculation using
commutativity in the right square and the triangle identities.

So from the data in (a) we may obtain that in (c), and the assignation
is injective, since $b$ is uniquely determined by universality of $a$
and commutativity on the right of~\eqref{eq:4}. For surjectivity,
suppose given $a$ and $b$ as in (c); we must show that $a$ exhibits
$Z$ as a $\varphi$-weighted colimit of $F$, in other words, that for
each $C \in \C$, the map~\eqref{eq:1} exhibits $\C(Z,C)$ as $[\varphi,
\C(F,C)]$, or in other words, that for each map $f \colon \varphi
\otimes_\A K \to \C(F,C)$, there is a unique map $\bar f \colon K \to
\C(Z,C)$ such that $f = \mu \circ (a \otimes_\A \bar f) \colon \varphi
\otimes_\A K \to \C(F,Z) \otimes_\A \C(Z,C) \to \C(F,C)$.
For existence, we let $\bar f$ be the
\begin{equation}
K \cong \A \otimes_\A K \xrightarrow{\eta \otimes_\A 1} \psi
\otimes_\B \varphi \otimes_\A K \xrightarrow{b \otimes_\B f} \C(Z,F)
\otimes_\B \C(F,C) \xrightarrow{\mu} \C(Z,C)\rlap{ ;}\label{eq:7}
\end{equation}
now rewriting with the right-hand square of~\eqref{eq:4} and using
the triangle identities and $F$'s preservation of units shows
that $f = \mu \circ (a \otimes_\A \bar f)$. For uniqueness, let $g
\colon K \to \C(Z,C)$ also satisfy $f = \mu \circ (a \otimes_\A g)$.
Substituting into~\eqref{eq:7} shows that $\bar f$ is the composite
\[
K \cong \A \otimes_\A K \xrightarrow{\eta \otimes_\A 1} \psi
\otimes_\B \varphi \otimes_\A K \xrightarrow{b \otimes_\B a \otimes_\A
g} \C(Z,F)
\otimes_\B \C(F,Z) \otimes_A \C(Z,C) \xrightarrow{\mu} \C(Z,C)\rlap{
  ;}
\]
which by rewriting with the left square of~\eqref{eq:4} and using
$Z$'s preservation of identities is equal to $g$. This proves the
equivalence (a) $\Leftrightarrow$ (c); now (a) $\Leftrightarrow$ (b) follows by duality.
\end{proof}

\subsection{Examples} We first consider examples wherein $\W$ is the
one-object bicategory corresponding to a monoidal category $\V$.
\begin{itemize}
\item Let $\V = \cat{Set}$, and let $\varphi$ be the weight for
  splittings of idempotents. The result recovers the bijection, for an
  idempotent $e \colon A \to A$, between: maps $p \colon A \to B$ coequalising
  $e$ and $1_A$; maps $i \colon B \to A$ equalising $e$ and $1_A$; and
  pairs $(i,p)$ with $pi = 1_A$ and $ip = e$.\vskip-0.25\baselineskip
\item Let $\V = \cat{Set}_\ast$, and let $\varphi$ be the weight for
  an initial object. The result recovers the bijection in a pointed
  category between: initial objects; terminal objects; and objects $X$
  with $1_X = 0_X$.\vskip-0.25\baselineskip
\item Let $\V = \cat{Ab}$, and let $\varphi$ be the weight for binary
  coproducts. The result recovers the bijection, for objects $A, B$ in
  a pre-additive category, between: coproduct diagrams $i_1
  \colon A \to Z \leftarrow B \colon i_2$; product diagrams $p_1
  \colon A \leftarrow Z \rightarrow B \colon p_2$; and tuples $(i_1,
  i_2, p_1, p_2)$ such that $p_ji_k = \delta_{ik}$ and $i_1p_1 +
  i_2p_2 = 1_Z$.
\item Let $\V = \bigvee\text-\cat{Lat}$, and let $\varphi$ be the
  weight for $J$-fold coproducts (for $J$ a small set). The result
  recovers the bijection, for objects $(A_j : j \in J)$ in a
  sup-lattice enriched category, between: coproduct diagrams $(i_j
  \colon A_j \to Z)_{j \in J}$; product diagrams $(p_j \colon Z \to
  A_j)_{j \in J}$; and families $(i_j)_{ j \in J}$ and $(p_j)_{j \in
    J}$ with $p_j i_k = \delta_{jk}$ and $\bigvee_j i_j p_j = 1_Z$.
\item Let $\V= k\text-\cat{Vect}$ for $k$ a field of characteristic
  zero, let $G$ be a finite group, and let $\varphi \colon k \tor kG$
  be the trivial right $kG$-module $k$. By Burnside's Lemma, $\varphi$
  has right adjoint $kG \tor k$ given by the trivial left $kG$-module
  $k$. Now the result recovers the bijection, for a $G$-representation
  $A$ in a $k$-linear category, between: maps $p \colon A \to Z$
  exhibiting $Z$ as an object of coinvariants of $A$; maps $i \colon Z
  \to A$ exhibiting $Z$ as an object of invariants of $A$; and pairs
  of maps $(i,p)$ with $pi=1_Z$ and $ip = \tfrac{1}{\abs G} \Sigma_{g
    \in G}\,g $.
\end{itemize}
We conclude with two examples where $\W$ is a genuine bicategory.
\begin{itemize}
\item Let $(\C, j)$ be a subcanonical site, and let $\W$ denote the
  full sub-bicategory of $\cat{Span}(\cat{Sh}(\C))^\op$ on objects of the
  form $\C(\thg, X)$. To any prestack $p \colon \E \to \C$ over $\C$,
  we may (as in~\cite{Betti1983Variation}) associate a $\W$-category
  with objects those of $\E$, extents $\epsilon(a) = p(a)$, and
  hom-object from $a$ to $b$ given by the span $\C(\thg, pa)
  \leftarrow \E(a,b) \rightarrow \C(\thg, pb)$ in $\cat{Sh}(\C)$;
  here $\E(a,b)(x)$ is the set of all triples $(f,g,\theta)$ with $f
  \colon pa \leftarrow x \rightarrow pb \colon g$ in $\C$ and $\theta
  \colon f^\ast(a) \to g^\ast(b)$ in $\E_x$ (note that $\E(a,b)$ is a
  sheaf by the prestack condition).

  For any cover $(f_i \colon U_i \to U)_{i \in I}$ in $\C$, we have a
  $\W$-category $R[f]$ with object set $I$, extents $\epsilon(i) =
  U_i$ and hom-objects $R[f](j, i) = \C(\thg, U_j) \leftarrow \C(\thg,
  U_j \times_U U_i) \rightarrow \C(\thg, U_i)$. There is a profunctor
  $\varphi \colon U \tor R[f]$ with components given by the spans
  $\varphi(i, \star) = \C(\thg, U_i) \leftarrow \C(\thg, U_i)
  \rightarrow \C(\thg, U)$. Writing $\psi \colon R[f] \tor U$ for the
  reverse profunctor, it is not hard to see that $\varphi \dashv \psi$
  (in fact they are adjoint pseudoinverse).

  The result now says the following. Given a prestack $p \colon \E \to
  \C$, a cover $(f_i \colon U_i \to U)$ in $\C$, and a family of spans
  $p_{ij} \colon a_i \leftarrow a_{ij} \to a_j \colon q_{ij}$ in $\E$
  whose legs are cartesian over the projections $U_i \leftarrow U_i
  \times_U U_j \to U_j$, there is a bijection between: cocones $(h_i
  \colon a_i \to a)$ in $\E$ over the $f_i$'s that are colimiting for
  the diagram comprised of the $p_{ij}$'s and $q_{ij}$'s; universal
  objects $a \in \E_U$ equipped with vertical maps $f_i^\ast(a) \to
  a_i$ fitting into double pullback squares
\[
\cd{
  f^\ast_i(a) \ar[d] & \ar[l] \cdot \ar[d] \ar[r] & f^\ast_j(a) \ar[d] \\
  a_i & a_{ij} \ar[l]^{p_{ij}} \ar[r]_{q_{ij}} & a_j\rlap{ ;}
}
\]
and objects $a \in \E_U$ equipped with a family of maps $(h_i \colon
a_i \to a)$ cartesian over the $f_i$'s. This
generalises~\cite[Proposition~5.2(b)]{Street1983Enriched}\footnote{The
proposition numbering here is taken from the TAC reprint.}.

\item Let $\W$ denote the bicategory whose objects are sets, and whose
  hom-category $\W(X,Y)$ is the category of finitary functors
  $\cat{Set} / Y \to \cat{Set} / X$; note that $\W(X,Y) \simeq
  [\cat{Fam}(Y) \times X, \cat{Set}]$, where $\cat{Fam}(Y)$ has as
  objects, finite lists of elements of $Y$, and as maps $(y_0, \dots,
  y_m) \to (z_0, \dots, z_n)$, functions $f \colon [m] \to [n]$ such
  that $y_{i} = z_{f(i)}$. To any cartesian multicategory $M$ (i.e., a
  \emph{Gentzen multicategory} in the sense
  of~\cite{Lambek1989Multicategories}) we may associate a
  $\W$-category $\M$ whose objects of extent $X$ are $X$-indexed
  families of objects of $M$, and whose hom-object between families
  $(a_x)_{x \in X}$ and $(b_y)_{y \in Y}$ is the presheaf
\[
\M((b_y), (a_x))(y_0, \dots, y_m; x) = M(b_{y_0}, \dots, b_{y_m}; a_x)
\]
in $[\cat{Fam}(Y) \times X, \cat{Set}]$; reindexing along maps in $Y$
makes use of the cartesianness of the multicategory structure.
Composition and units in $\M$ follow from those in $M$.

Given a finite set $X = \{x_0, \dots, x_n\}$, let $\varphi \colon 1
\tor X$ be the $\W$-profunctor whose unique component is the
representable $y(x_0, \dots, x_n; \star) \in [\cat{Fam}(X) \times 1,
\cat{Set}]$. This has a right adjoint $\psi \colon X \tor 1$ whose
unique component is the presheaf $\Sigma_{x \in X} y(\star; x) \in
[\cat{Fam}(1) \times X, \cat{Set}]$. The result now establishes a
bijection, for any finite family $(a_0, \dots, a_n)$ of objects in a
cartesian multicategory $M$, between data of the following three
forms: first, an object $a$ and a multimap $i \in M(a_0, \dots, a_n;
a)$, composition with which induces bijections between $M(b_0, \dots,
b_k, a, c_0, \dots, c_\ell; d)$ and $M(b_0, \dots, b_k, a_0, \dots,
a_n, c_0, \dots, c_\ell; d)$; second, an object $a$ and unary maps
$p_j \in M(a; a_j)$, composition with which establishes bijections
between $M(b_0, \dots, b_k; a)$ and $\Pi_j M(b_0, \dots, b_k; a_j)$;
third, an object $a$ and maps $i$ and $p_j$ as above such that $p_j
\circ i = \pi_j \in M(a_0, \dots, a_n; a_j)$ and $i \circ (p_0, \dots,
p_n) = 1_a \in M(a;a)$. This
generalises \cite[Proposition~3.5]{Garner2014Lawvere}.

\end{itemize}

\bibliography{bibdata}

\begin{thebibliography}{1}

\bibitem{Betti1983Variation}
{\sc Betti, R., Carboni, A., Street, R., and Walters, R.}
\newblock Variation through enrichment.
\newblock {\em Journal of Pure and Applied Algebra 29}, 2 (1983), 109--127.

\bibitem{Garner2014Lawvere}
{\sc Garner, R.}
\newblock Lawvere theories, finitary monads and {C}auchy-completion.
\newblock {\em Journal of Pure and Applied Algebra 218}, 11 (2014), 1973--1988.

\bibitem{Lambek1989Multicategories}
{\sc Lambek, J.}
\newblock Multicategories revisited.
\newblock In {\em Categories in computer science and logic ({B}oulder, 1987)},
  vol.~92 of {\em Contemporary Mathematics}. American Mathematical Society,
  1989, pp.~217--239.

\bibitem{Pare1971On-absolute}
{\sc Par{{\'e}}, R.}
\newblock On absolute colimits.
\newblock {\em Journal of Algebra 19\/} (1971), 80--95.

\bibitem{Street1983Absolute}
{\sc Street, R.}
\newblock Absolute colimits in enriched categories.
\newblock {\em Cahiers de Topologie et Geom\'{e}trie Diff\'{e}rentielle
  Cat\'{e}goriques 24}, 4 (1983), 377--379.

\bibitem{Street1983Enriched}
{\sc Street, R.}
\newblock Enriched categories and cohomology.
\newblock {\em Quaestiones Mathematicae 6\/} (1983), 265--283.
\newblock Republished as: {\emph{Reprints in Theory and Applications of
  Categories 14}} (2005).

\end{thebibliography}

\end{document}